\documentclass{amsart}
\usepackage[utf8]{inputenc}
\usepackage[english]{babel}
\usepackage{amssymb,amsmath,amsthm}
\usepackage{verbatim}

\usepackage{color}
\usepackage{amssymb}
\usepackage{amsthm,array,amssymb,amscd,amsfonts,latexsym, url}
\usepackage{amsmath}

\usepackage[all]{xy}
\numberwithin{equation}{section}

\newtheorem{dummy}{dummy}[section]

\newtheorem{question}[dummy]{Question}
\newtheorem{theorem}[dummy]{Theorem}
\newtheorem{corollary}[dummy]{Corollary}
\newtheorem{lemma}[dummy]{Lemma}
\newtheorem{proposition}[dummy]{Proposition}
\newtheorem{remark}[dummy]{Remark}

\newtheorem{theo}{Theorem}
\newtheorem{prop}[theo]{Proposition}

\newtheorem{coro}[theo]{Corollary}
\newtheorem{rema}[theo]{Remark}

\newfont{\gothic}{eufb10}

\def\A{\mathbb A}

\def\L{\mathbb L}
\def\P{\mathbb P}

\def\R{\mathbb R}

\def\Z{\mathbb Z}

\def\OO{\mathcal O}

\def\XX{\mathcal X}

\def\wt{\tilde}

\def\Spec{{\operatorname{Spec}}}

\def\={\;=\;}
\def\bal{\begin{aligned}}
\def\eal{\end{aligned}}
\def\be{\begin{equation}\label}
\def\ee{\end{equation}}
\def\mod#1{\; \left({\rm mod} \; #1\right)}

\title{Variation of Stable Birational Types of Hypersurfaces}
\author{Evgeny Shinder, with an Appendix by Claire Voisin}
\address{E.S.: School of Mathematics and Statistics, University of Sheffield,
Hounsfield Road, S3 7RH, UK, and
National Research University Higher School of Economics, Russian Federation}
\email{e.shinder@sheffield.ac.uk}

\address{C.V.: Coll\`{e}ge de France
3 rue d'Ulm, 75005 Paris, France}
\email{claire.voisin@imj-prg.fr}

\begin{document}

\maketitle

\begin{abstract}
We introduce and study the question
how can stable birational types
vary in a smooth proper family.
Our starting point is the specialization for stable birational types of 
Nicaise and the author and our emphasis is on stable
birational types of hypersurfaces.
Building up on the work of Totaro 
and Schreieder 
on stable irrationality
of hypersurfaces of high degree, we
show that smooth Fano hypersurfaces of large degree 
over a field of characteristic zero are in general not stably birational to each other. 
In the appendix Claire Voisin proves a similar result in a different setting using the Chow decomposition of diagonal and unramified cohomology.    
\end{abstract}

\section{Introduction}

Let $k$ be an uncountable algebraically closed field of characteristic zero. 
Recall that $k$-varieties $X$, $Y$ of the same dimension are 
called \emph{stably birational} if $X \times \P^m$ and $Y \times \P^m$ are birational 
for some $m \ge 0$.
If in the above definition $Y$ is a projective space, then $X$ is called
\emph{stably rational}. 
There has been recently a lot of progress in showing that for large
classes of varieties, including Fano 
hypersurfaces of high degree, 
very general members are stably irrational \cite{Voisin, CTP,
Totaro, Schreieder}.
In this paper we introduce and study the following more general question:

\begin{question}
Given a family of smooth projective varieties,
how can we decide if all members 
are stably birational to each other?
\end{question}

We answer this question for Fano hypersurfaces 
of sufficiently high degree.
Our main result is the following:

\begin{theorem}[See Theorem \ref{thm:hypers}]
If there exists a stably irrational smooth projective hypersurface
of dimension $n$ and degree $d \le n+1$, 
then very general hypersurfaces of dimension $n$ and degree $d$
are not stably birational to each other.
\end{theorem}

Here by very general hypersurfaces we 
mean pairs
of hypersurfaces corresponding to points in the parameter
space $\P(H^0(\P^{n+1}, \OO(d)))^2$ lying in the complement
of a countable union of divisors.

Thus the only case when smooth hypersurfaces of given
degree and dimension are stably birational to each other
is when they are all stably rational.
This happens in degrees one and two, and for
cubic surfaces, and it is widely expected that 
no other such cases exist.

It has been proved by Totaro \cite{Totaro}
that in every dimension $n \ge 3$
very general Fano hypersurfaces of degree $d \ge 2\lceil\frac{n+2}{3}\rceil$ are stably irrational. 
Schreieder improved Totaro's bound to
$d \ge \log_2(n) + 2$ \cite{Schreieder}.
Using \cite[Corollary 1.2]{Schreieder} and the Theorem above we deduce the following.

\begin{corollary}
For $n \ge 3$ and $d \ge \log_2(n) + 2$, very general
hypersurfaces of dimension $n$ and degree $d$ are not stably birational
to each other.
\end{corollary}

In particular we see that there are uncountably many stable birational
types of such hypersurfaces.
The first interesting case when Corollary applies
is that of quartic threefolds ($n = 3$, $d = 4$, here stable irrationality of the very general member follows from \cite{CTP}).
%Nothing is known about the case of 
%stable rational types of cubic threefolds.

Under the assumptions of the Theorem every stable birational
type is attained at a countable union of Zariski closed subsets in the parameter space of smooth hypersurfaces.
A more explicit description of which hypersurfaces of fixed
dimension and degree would be stably birational to the given one, seems
completely out of reach. 
%For instance, one would probably expect
%that the locus of stably rational hypersurfaces may be 
%in some sense larger than the locus of any other stable birational type
%in the family but nothing definitive can be said at the moment.

Our approach to stable birational types relies on the Grothendieck ring of varieties, the Larsen-Lunts Theorem \cite{LL}
and the specialization map \cite{NS, KT}.
Firstly, we reformulate results of \cite{NS} 
by introducing the idea of a variation of stable birational
types and show that if stable birational type in a family is not
constant, then it has to vary in a strong sense (Theorem \ref{thm:variation}). 
Then, by constructing an appropriate
degeneration of smooth hypersurfaces to a hyperplane
arrangement, with desingularized total space (Lemma \ref{lemma:model}) 
and showing that
the class of this hyperplane arrangement in the Grothendieck ring
is congruent to $1$ modulo $\L$ (Lemma \ref{lem:hyperp})
we deduce that under the conditions
of the theorem, stable birational types of hypersurfaces can not be constant
(Theorem \ref{thm:hypers}).
The same method would apply to any family that has a 
smooth stably
irrational member alongside a smooth 
stably rational member, or more generally,
a member with mild singularities and whose class in the Grothendieck
modulo $\L$ is equal to one, and
provided that the total space of the degeneration
is smooth or has mild singularities.

%Let us note that no other existing method is
%adapted to distinguish
%stable birational types. Indeed, cohomological invariants
%are typically only catching stably rational vs stably irrational types (see %however
%the Appendix to this paper by Claire Voisin
%where decomposition of diagonal is used
%to distinguish stable birational types),
%whereas methods such as birational rigidity and intermediate jacobians
%deal only with birationality, but do not control stable birationality.

In addition to using the Grothendieck ring of varieties and
the specialization map, one novelty of this work is
making use of degeneration
of a hypersurface to a hyperplane arrangement.
Such degenerations are ubiquitous in algebraic
geometry, starting from computing the genus of a plane curve
and all the way
to the modern Gross-Siebert program. 
These degenerations also played their role
in rationality problems \cite{CTO}.
Our contribution however is the direct
link between having a semistable
fiber in a family and variation of
the stable birational types
of the smooth fibers.
One familiar example of this behaviour
is that an isotrivial elliptic surface 
can not have semistable fibers.
This well-known fact is an easy corollary
of Proposition \ref{prop:snc}.

In the Appendix to this paper Claire Voisin proves a similar result
regarding variation of stable birational types in a slighly different setting
using decomposition of diagonal and unramified cohomology.
Very soon after appearance of this work, Stefan Schreieder gave a different
proof of Corollary 1.3 using degeneration to hyperplane arrangement and decomposition of the diagonal, also relying on 
\cite{Schreieder-var}; in fact
Schreieder's proof
does not use resolution of singularities and thus 
generalizes the statement to a field of an arbitrary characteristic.

Finally we note that
unlike in Hodge theory, 
where the term ``variation"
can be understood using the period map
between moduli spaces, 
our
term ``variation of stable birational types" 
has a very naive meaning;
it is not at all clear
how one could introduce a reasonable moduli space
of stable birational types.

\subsection*{Acknowledgements}

The author would like to thank Adel Betina, 
Christian B\"ohning,
Jean-Louis Colliot-Th\'el\`ene,
Sergey Galkin, Alexander Kuznetsov, 
Johannes Nicaise, Alexander Pukhlikov, Claire Voisin, Stefan Schreieder, Konstantin Shramov
for discussions and encouragement.
The idea of using degeneration to a hyperplane
arrangement, as opposed to a nodal hypersurface is due to 
an e-mail correspondence with Sergey Galkin.
A suggested simplification in the proof of Lemma \ref{lemma:model}, as explained
in Remark \ref{rem:log} is due to the referee.

The author was partially supported by 
Laboratory of Mirror Symmetry NRU HSE, RF government grant, ag. N~14.641.31.0001.

\medskip

%The paper is dedicated to Miles Reid who 
%inspired and educated 
%several generations of algebraic geometers
%through direct contact as well as with his books and %papers, which
%often seem to be intentionally written to inspire and to %educate.
%In 2014 Miles told me that at some point young people %should
%stop dreaming and do something
%useful instead, and I try to follow this advice ever since.

\medskip

\subsection*{Notation}
By a variety we mean a separated irreducible and reduced scheme
of finite type over $k$.
By a point of a variety we a mean a closed point.
We say that a property holds for very general points of a variety
if it holds away from a countable union of divisors.

\section{Preliminary results}

\subsection{Grothendieck ring of varieties}

Recall that the Grothendieck ring of varieties $K_0(Var/k)$
is generated as an abelian group by isomorphism 
classes $[X]$ of schemes of finite type $X/k$ modulo the scissor relations
\[
[X] = [U] + [Z]
\]
for every closed $Z \subset X$ with open complement $U \subset X$.
The product structure on $K_0(Var/k)$ is induced by product of schemes. We write $\L \in K_0(Var/k)$
for the class of the affine line $[\A^1]$.

The following lemma is useful when degenerating smooth varieties to
hyperplane arrangements.

\begin{lemma}\label{lem:hyperp}
Let $H_1, \dots, H_r \subset \P^{n+1}$ be a collection of distinct 
hyperplanes in $\P^{n+1}$ such that $\bigcup_{i=1}^r H_i$ is a simple normal crossing divisor,
that is we assume that any intersection of $k$ hyperplanes is either empty or of codimension
$k$. Then we have
\[\;
[H_1 \cup \dots \cup H_r] = \sum_{j=0}^n (-1)^j \binom{r}{j+1} [\P^{n-j}],
\]
and if $r \le n + 1$, then $[H_1 \cup \dots \cup H_r] \equiv 1 \mod{\L}$.
\end{lemma}
\begin{proof}
Let $P_{r,n} \in K_0(Var/k)$ be the class of a simple normal crossing hyperplane arrangement of $r$ hyperplanes in $\P^{n+1}$
in the Grothendieck ring of varieties. It follows from the inductive argument
below that the class $P_{r,n}$ only depends on $r$ and $n$ and not on the relative
positions of the hyperplanes.

We prove the formula for $P_{r,n}$ for
$r \ge 1$, $n \ge 0$ using induction.
For the induction base we have for all $r \ge 1$,
$P_{r,0} = r$ ($r$ points in $\P^1$)
and for all $n \ge 0$ we have $P_{1,n} = [\P^n]$. 
We assume that the formula is true for $P_{r,n-1}$ and $P_{r-1,n-1}$.
Given $r \ge 2$ hyperplanes in $\P^{n+1}$, intersecting the first 
$r - 1$ of them with the last one, gives rise to
an arrangement of $r-1$ 
hyperplanes in $\P^n$, which is still simple normal 
crossing. Using inclusion-exclusion
we obtain
\[
P_{r,n} = P_{r-1,n} + [\P^n] - P_{r-1,n-1},
\]
which by induction hypothesis can be rewritten as
\[
P_{r,n} = \sum_{j=0}^n (-1)^j \binom{r-1}{j+1} [\P^{n-j}]
+[\P^n] - \sum_{i=0}^{n-1} (-1)^i \binom{r-1}{i+1} [\P^{n-i-1}]
\]
which easily gives the desired result.

Finally, if $r \le n+1$, then 
\[
P_{r,n} \equiv \sum_{j=0}^n (-1)^j \binom{r}{j+1} = 
\sum_{i=1}^{r} (-1)^{i-1} \binom{r}{i} = 
1 \mod{\L}.
\]
\end{proof}

\subsection{Resolution of one toric singularity}

When constructing resolutions of singularities
for the total space of a degeneration of hypersurfaces
the following result is useful.
We refer to \cite{Fulton} for standard facts and constructions
from toric geometry.

\begin{lemma}\label{lem:desing}
Let be $\XX$ be a hypersurface in $\A^{n+2}$
defined by equation 
\[
t \cdot y = z_1 \cdots z_n
\]
and let $\pi: \XX \to \A^1$ be the morphism given by the $t$ coordinate.

(1) Let $N = \Z^{n+1}$ with the standard basis $e_1, \dots, e_{n+1}$
and let $N_\R = N \otimes \R$.
For every $1 \le i \le n$ let $f_i = e_i + e_{n+1} \in N$.
Let $\sigma \subset N_\R = \R^{n+1}$ be the cone
generated by the vectors $e_1, \dots, e_n$, $f_1, \dots, f_n$.
Then $\XX$ is the toric variety corresponding to the cone $\sigma$,
that is $\XX = \Spec(k[N^\vee \cap \sigma^\vee])$.

(2) Subdivision of $\sigma$ into $n$ cones
\[
\sigma_k := 
\R_{\ge 0} f_1 + \dots + \R_{\ge 0} f_k + \R_{\ge 0} e_k + \dots + 
\R_{\ge 0}e_n \subset N_\R, \quad
1 \le k \le n
\]
provides a resolution of singularities $\tau: \wt{\XX} \to \XX$.
The composition $\wt{\pi} := \pi \circ \tau$ has a
reduced simple normal crossing fiber over $0 \in \A^1$.

(3) Explicitly desingularization $\tau$ is obtained by a sequence 
of blow ups of proper preimages of $n-1$ Weil divisors 
$V(t, z_1), \dots, V(t,z_{n-1}) \subset \XX$.
\end{lemma}
\begin{proof}
The proof is a standard computation in toric geometry.

(1) Let $M = N^\vee$ be the dual lattice
with the dual basis $e_1^*, \dots, e_{n+1}^*$.
The dual cone $\sigma^\vee \subset M_\R$
is described by the system of inequalities for $(a_1, \dots, a_{n+1}) \in M_\R$:
\[
a_1 \ge 0, \; \dots, \; a_n \ge 0, 
\]
\[
a_1 + a_{n+1} \ge 0, \; \dots, \; a_n + a_{n+1} \ge 0.
\]

It is clear that the $n+2$ vectors
\[
e_1^*, \dots, e_n^*, e_{n+1}^*, e_1^* + \dots + e_n^* - e_{n+1}^* \in M
\]
all satisfy these inequalities, and every integral point in $\sigma^\vee$
can be written as a non-negative integer combination of these vectors
(indeed, if $a_{n+1} \ge 0$, then we are done, while if $a_{n+1} < 0$, 
all other coordinates must be positive and a multiple of
$e_1^* + \dots + e_n^* - e_{n+1}^*$ can be subtracted).

If we set 
\[
z_1, \dots, z_n, t, y
\]
to be the monomials corresponding to the vectors above, they satisfy
a single relation $ty = z_1 \cdots z_n$.

(2) The cone $\sigma$ combinatorially is a cone
over a prism $\Delta^{n-1} \times [0,1]$ ($\Delta$ is a simplex), 
and the subdivision we consider corresponds to
a standard subdivision of this prism into $n$ simplices.
To describe this construction in detail note that we have seen that the dual
cone $\sigma^\vee$ is generated by $e_1^*, \dots, e_n^*, e_{n+1}^*, e_1^* + \dots + e_n^* - e_{n+1}^* \in M$, so that the cone $\sigma \subset N_\R$
is the set of solutions of
\[
x_1 \ge 0, \; \dots, \; x_{n+1} \ge 0, \quad x_{n+1} \le x_1 + \dots + x_n.
\]
For every $1 \le k \le n$ let us consider the cones given by 
\[
x_1 \ge 0, \; \dots, \; x_{n+1} \ge 0, \quad x_1 + \dots + x_{k-1} \le
x_{n+1} \le x_1 + \dots + x_k.
\]
These cones obviously form a partition of $\sigma$ and it is easy
to see that these are precisely the cones $\sigma_k$ with boundary
rays generated by $f_1, \dots, f_k, e_k, \dots, e_n$.

The new fan, consisting of the cones $\sigma_k$ and all their faces
has its cones generated by basis vectors of the lattice $N$, hence the corresponding
morphism $\tau: \wt{\XX} \to \XX$ is a resolution of singularities.

To check that the fiber $\wt{\pi}$ is reduced 
simple normal crossing over $0 \in \A^1$, we
consider each affine toric chart $U_k$, corresponding to the cone $\sigma_k$.
By our choice of coordinates the restriction of $\wt{\pi}$ to $U_k$
corresponds to the projection onto the last coordinate 
$N_\R = \R^{n+1} \to \R$. Fiber over $0 \in \A^1$ being a 
reduced simple normal crossing divisor
in $U_k$ translates into the fact that every vector
$f_1, \dots, f_k, e_k, \dots, e_n$ has zero or one as its last coordinate.

(3) Let us describe the effect of the blow up of $V(t,z_1) \subset \XX$ on our toric model.
We have two open charts, on the first open chart which we call $U_1$ 
we have $z_1 = t z_1'$
so the coordinates are $t, y, z_1', z_2, \dots, z_n$
and the equation is
\[
y = z_1' z_2 \cdots z_n.
\]
On the other open chart which we call $\XX'$ 
we have $t = z_1 t'$ so coordinates are
$t', y, z_1, \dots, z_n$
and the equation is
\[
t'y = z_2 \cdots z_n.
\]
The gluing between the two open charts is $t' = \frac1{z_1'}$.
%In terms of the cone $\sigma^\vee$ described in the proof of part (1),
%the coordinates $z_1'$, $t'$ correspond to vectors
%\[
%e_1^* - e_{n+1}^*, \quad e_{n+1}^* - e_1^*
%\]
%respectively. 
We see that $U_1$ is the affine space
with coordinates $z_1', z_2, \dots, z_n$ which are monomials
corresponding to vectors
\[
e_1^* - e_{n+1}^*, e_2^*, \dots, e_n^*, e_{n+1}^*,
\]
which is precisely the generators for the dual cone to $\sigma_1$,
while $\XX'$ has coordinates corresponding to the monomials
\[
e_{n+1}^* - e_1^*, e_1^*, \dots, e_n^*, e_1^* + \dots + e_n^* - e_{n+1}^*,
\]
and it follows easily that $\XX'$ 
is the toric variety corresponding to the cone
\[
\sigma' := \sigma_2 \cup \dots \cup \sigma_n = \sum_{i=2}^n \R_{\ge 0} e_i
+ \sum_{i=1}^n \R_{\ge 0} f_i \subset N_\R.
\]
Furthermore $\XX'$ is the product of 
$\A^1$ ($z_1$ coordinate) with
the same model in $n-1$ variables.

Since $U_1$ is already smooth, it will not be affected
by further blow ups of divisors, 
while the proper preimage of $V(t,z_k)$ (for $k \ge 2$)
in $\XX'$ is $V(t',z_k)$, and the same argument can be applied
to $\XX'$ to get open charts $U_2, \dots, U_n$, corresponding
to the subdivision of $\sigma'$.

The process terminates after $n-1$ steps, when two smooth 
charts are produced.
\end{proof}

\section{Stable birational types of hypersurfaces}

\subsection{Variation of stable birational types}

We recall the following result of Larsen and Lunts which holds over 
arbitrary fields of characteristic zero and which
provides the link
between birational geometry and the Grothendieck ring of
varieties.

\begin{theorem}\cite{LL}\label{thm:LL}
If $X$ and $Y$ are smooth projective varieties with
classes $[X], [Y] \in K_0(Var/k)$, then $X$
and $Y$ are stably birational if and only if 
\[
[X] \equiv [Y] \mod{\L}.
\]
\end{theorem}

Thus for a smooth projective variety $X$ the element
$[X] \in K_0(Var/k)/(\L)$
encodes the stable birational class of $X$.
%More generally, for instance when $X$ is %projective but reducible, 
%we may think about $[X] \mod{\L}$ as an 
%analog of the stable birational class.

We now introduce the idea of the variation of stable birational types
in the smooth and simple normal crossing settings.
These rely on the results of \cite{NS}.

\begin{theorem}\label{thm:variation}
Let $S$ be a variety and let
$\pi: \XX \to S$ be a smooth 
proper morphism with connected fibers. Then one of the following is true:
\begin{enumerate}
    \item[(a)] {\bf Constant stable birational type:} all fibers $\pi^{-1}(t)$, 
    $t \in S$ are stably birational.
    \item[(b)] {\bf Variation of stable birational type:}
    for very general points $(t, t') \in S \times S$ the fibers $\pi^{-1}(t)$ and
    $\pi^{-1}(t')$ are not stably birational to
    each other.
\end{enumerate}
\end{theorem}

\begin{proof}
For $i = 1, 2$ let us write 
$p_i: S \times S \to S$ for the two projections.
Let $\pi_i: \XX_i \to S \times S$ 
denote the base change of $\pi$ by $p_i$. Thus $\pi_1, \pi_2$
are smooth proper morphisms. Let $Z \subset S \times S$ 
be the set of points where the fibers
of $\pi_1$ and $\pi_2$ are stably birational,
in other words $Z$ consists of points $(t_1, t_2)$
such that $\pi^{-1}(t_1)$ and $\pi^{-1}(t_2)$
are stably birational.

By
\cite[Theorem 4.1.4]{NS}, $Z$ is a countable union of Zariski closed
subsets of $S \times S$.
Thus either $Z = S \times S$, which corresponds 
to the case (a), or $Z \subsetneq S \times S$,
so that points in $S \times S \setminus Z$
are very general which corresponds to (b).
\end{proof}

The next Proposition provides a generalization of
the Theorem above to simple normal crossing singularities.
Instead of stable birational types we work
with classes in $K_0(Var/k)/(\L)$.

\begin{proposition}\label{prop:snc}
Let $C$ be a smooth connected curve and let 
$\pi: \XX \to C$ be a flat proper morphism with connected fibers and smooth total space $\XX$. 
Let $0 \in C$, and assume that
the restriction of $\pi$ to $C \setminus 0$
is smooth, and that $\pi^{-1}(0)$ is reduced simple
normal crossing.

If all fibers $\pi^{-1}(t)$ for $t \ne 0$
are stably birational to a smooth
projective variety $X$, then the class of the 
central fiber satisfies
\[\;
[\pi^{-1}(0)] \equiv [X] \mod{\L}.
\]
\end{proposition}
\begin{proof}
Let us form a constant family
$\pi': X \times C \to C$.
By assumption the two morphisms $\pi$,
$\pi'$ have stably birational fibers
for $t \ne 0$.

Thus using \cite[Proposition 4.1.1]{NS} we deduce that fibers over
$t = 0$ satisfy
\[
[\pi^{-1}(0)] \equiv [\pi'^{-1}(0)] = [X] \mod \L.
\]
\end{proof}

\subsection{Application to hypersurfaces}

In this section we study stably birational
types of hypersurfaces $X \subset \P^{n+1}$.
The interesting case is the Fano case, that
is the case when the degree of $X$ satisfies $d \le n+1$.

\begin{theorem}\label{thm:hypers}
Assume that there exists a 
smooth projective hypersurface
of dimension $n$ and degree $d \le n+1$
which is stably irrational. 
Then smooth projective
hypersurfaces
of dimension $n$ and degree $d$
admit a variation of stable birational types,
that is two very general such hypersurfaces
are not stably birational to each other.
\end{theorem}

\begin{remark}
By the main result of \cite{NS}, existence of a single
stably irrational smooth projective 
hypersurface of dimension $n$ and degree $d$
is equivalent to very general such hypersurfaces
being stably irrational.
\end{remark}

Before we prove the Theorem we need the following Lemma,
which provides a convenient degeneration of smooth 
hypersurfaces.

\begin{lemma}\label{lemma:model}
For every $n \ge 2$, $1 \le d \le n+1$ 
there exists a smooth connected curve $C$,
with a point  $0 \in C$, and a
flat proper morphism
\[
\pi: \XX \to C
\]
with smooth $\XX$
such that 
\begin{enumerate}
    \item All fibers $\pi^{-1}(t)$, for 
$t \ne 0$ are smooth projective hypersurfaces
of dimension $n$ and degree $d$,
\item The fiber $\pi^{-1}(0)$ is reduced simple
normal crossing and satisfies
\[
[\pi^{-1}(0)] \equiv 1 \mod{\L}.
\]
\end{enumerate}
\end{lemma}
\begin{proof}
We consider two sections $F_0, F_1 \in H^0(\P^{n+1}, \OO(d))$. We take $F_0$ to be a product of $d$
linearly independent linear forms,
and $F_1$ to be a general section. In particular
the hypersurface $V(F_1) \subset \P^{n+1}$ is smooth, 
and it intersects all the strata of the hyperplane
arrangement $V(F_0)$ transversally.

We set $\XX$ to denote the zero locus of $F_0 + tF_1$
in $\P^{n+1} \times \A^1$. After restricting to an open
subset $C \subset \A^1$ we may assume that the fibers
of $\pi: \XX' \to C$ for $t \ne 0$ are smooth.
The fiber $\pi^{-1}(0)$ is the hyperplane arrangement
$V(F_0)$. Since $d \le n+1$, Lemma \ref{lem:hyperp}
implies that the fiber $\pi^{-1}(0)$ satisfies
\[\;
[\pi^{-1}(0)] \equiv 1 \mod{\L}.
\]

Thus the morphism $\pi: \XX \to C$ satisfies
all the requirements of the Lemma except
for smoothness of the total space $\XX$.

We provide an explicit desingularization of $\XX$.
Let $E_1, \dots, E_d$ be the components of $\pi^{-1}(0)$.
We claim that blowing up the Weil divisors 
$E_1, \dots, E_n$ (in any order) produces a model satisfying
all the required properties.

Locally at every point $P \in \XX$ in the central fiber
$t = 0$, the model $\XX$ is given by equations of the
form
\begin{equation}\label{eq:local}
t \cdot f + l_1 \cdots l_d = 0,
\end{equation}
where $l_i$ are linear polynomials and $f$ is a
polynomial of degree $d$, in $n+1$ variables. 

We change coordinates so that $P = 0$, and by our transversality
assumptions the equation can be written as
\[
t \cdot (x_{n+1} + \text{terms of deg. $\ge 2$}) + x_1 \cdots x_k \cdot 
g(x_1, \dots, x_{n+1}) = 0,
\]
where $g(0) \ne 0$ and $k \le n$.
Taking formal completion of $\XX$ at $P$ we can change the coordinates
again to rewrite the defining local equation as
\[
t \cdot x_{n+1} = x_1 \cdots x_k.
\]
According to Lemma \ref{lem:desing} such singularities
are resolved by a sequence of blows up of Weil divisors $V(t,x_i)$ (which are precisely the components of
the central fiber containing the point $P$)
and this new model is semistable over $0 \in \A^1$ (and the rest of
the fibers are unchanged, so they are smooth hypersurfaces).

%Blow ups commute with the completion, and
%in terms of the original equation (\ref{eq:local})
%we have to blow up the sequence of smooth Weil divisors 
%\[
%E_1 = V(t, l_1), \dots, E_{d-1} = V(t, l_{d-1}). 
%\]
Since each open chart of the blow up is
a hypersurface in $\A^{n+2}$, the resulting blow ups
only glue in
Zariski locally trivial $\P^1$-fibrations. In particular, the class
of the central fiber in $K_0(Var/k)/(\L)$ 
does not change at each blow up.
\end{proof}

\begin{proof}[Proof of Theorem \ref{thm:hypers}]
Let $U \subset \P(H^0(\P^{n+1}, \OO(d)))$ be the open
subset parametrizing smooth hypersurfaces.
By Theorem \ref{thm:variation}, if stable
birational types of hypersurfaces of dimension $n$
and degree $d$, does NOT vary, it has be constant,
that is all such smooth hypersurfaces are
stably birational to a smooth projective 
variety $X$. 

We now consider the family $\pi: \XX \to C$ given by Lemma
\ref{lemma:model}. From what we explained above,
all fibers $\pi^{-1}(t)$, for $t \ne 0$ have to be
stably birational to $X$. By Proposition \ref{prop:snc},
the special fiber has to satisfy 
\[\;
1 \equiv [\pi^{-1}(0)] \equiv [X] \mod{\L}.
\]
This is a contradiction, since Larsen-Lunts Theorem \ref{thm:LL}
implies that $X$ is stably rational, contrary to our assumptions.
 \end{proof}

\begin{remark}\label{rem:log}
Using motivic volume expressed in terms of log smooth
models 
\cite[Appendix A]{NS} 
the result of Theorem \ref{thm:hypers} can
be obtained
without explicit resolution of singularities
of the model
by applying \cite[Theorem A.3.9]{NS}
to the appropriate log scheme.
However, the explicit resolution obtained in Lemma \ref{lemma:model}
can be useful for other purposes, such
as 
in the proof of the same
result in positive characteristic by
Schreieder
\cite[Theorem 5.1]{Schreieder-var}.
\end{remark}
 
\providecommand{\arxiv}[1]{{\tt arXiv:#1}}

%\footheight 35cm

%\begin{comment}

%\title{Appendix: Stable birational equivalence and decomposition of the diagonal}
%\author{Claire Voisin}

%\tableofcontents

%\date{}

\section*{Appendix: Stable birational equivalence and decomposition of the diagonal,
by Claire Voisin}

We prove in this appendix  that, if  a family of  projective varieties has a mildly singular member
with a nonzero unramified cohomology class with given coefficients, while the very general member $Y$ is smooth and
has no such class, the stable birational equivalence class of the fibers $Y_t$ is not constant.
In particular, quartic and sextic double covers of $\mathbb{P}^3$  do not have a constant stable birational type.
This result is inspired by the main theorem of Shinder in this paper. Note however that the assumptions
and range of applications in both statements are different. We will work  over any algebraically closed field $k$ of infinite transcendence degree over the prime field but  the main  application   (Theorem \ref{theoappli})  will assume    characteristic $0$. We refer to Schreieder recent note \cite{schreiedervar} for generalizations and
a similar statement in nonzero characteristic.

We start with the following decomposition of the diagonal result for stable birational equivalence.
\begin{prop}\label{propcharstabdecomp}  Let $X$, $Y$ be two smooth projective varieties of dimension $n$. Assume
$X$ and $Y$ are stably birational.
Then there exist codimension $n$ cycles
$$\Gamma\in{\rm CH}^n(X\times Y),\,\,\Gamma'\in{\rm CH}^n(Y\times X)$$
such that
\begin{eqnarray}
\label{eqbiratdecomp}
\Gamma'\circ \Gamma =\Delta_X+Z_X\,\,{\rm in}\,\,{\rm CH}^n(X\times X),\\
\nonumber
\Gamma\circ \Gamma' =\Delta_Y+Z_Y \,\,{\rm in}\,\,{\rm CH}^n(Y\times Y),
\end{eqnarray}
where $Z_X$ is supported on $D_X\times X$ for some proper closed algebraic subset $D_X\subset X$, and
 $Z_Y$ is supported on $D_Y\times Y$ for some proper closed algebraic subset $D_Y\subset Y$.
\end{prop}
\begin{proof}  When $ X$ and $Y$ are actually birational, this statement is proved in \cite{CTV}. In this case, we simply take for  $\Gamma$
the graph of a birational map
$\phi:X\dashrightarrow Y$ and for $\Gamma'$ the graph of $\phi^{-1}$. The equality
$\Gamma'\circ \Gamma =\Delta_X$ (resp. $\Gamma\circ \Gamma' =\Delta_Y$) is  in this case satisfied at the level of cycles  on
$U\times X$, resp. $V\times Y$, where $U\cong V$ is a Zariski open set of $X$ on which $\phi$ is an isomorphism onto its image $V\subset Y$.
Assume now that
$$\phi: X\times \mathbb{P}^r\dashrightarrow Y\times \mathbb{P}^r$$
is a birational map for some $r$. Then by the previous step,
there exist
$$\Gamma_\phi \in {\rm CH}^{n+r}(X\times  \mathbb{P}^r\times Y\times  \mathbb{P}^r),\,\,
\Gamma'_{\phi}\in {\rm CH}^{n+r}(Y\times  \mathbb{P}^r\times X\times  \mathbb{P}^r)$$
such that formulas  (\ref{eqbiratdecomp})  hold  for some proper closed algebraic subsets
$D\subset X\times   \mathbb{P}^r$, resp. $D'\subset  Y\times   \mathbb{P}^r$.
For any point $O\in  \mathbb{P}^r$, define
\begin{eqnarray} \label{eqdefGamma}\Gamma:=p_{XY*}({\Gamma_{\phi}}_{\mid X\times O \times Y\times  \mathbb{P}^r})
\\
\nonumber
 \Gamma':=p_{YX*}({\Gamma'_{\phi}}_{\mid Y\times O \times X\times  \mathbb{P}^r}),
\end{eqnarray}
where $p_{XY}$ is the projection from $X\times O  \times Y\times  \mathbb{P}^r$ to $X\times Y$, and $p_{YX}$
 is the projection from $Y\times O  \times X\times  \mathbb{P}^r$ to $Y\times X$.
We have to show that (\ref{eqbiratdecomp}) holds.

Let us decompose ${\rm CH}(X\times \mathbb{P}^r\times Y\times \mathbb{P}^r)$ as polynomials in $h_1,\,h_2$ with coefficients in
${\rm CH}( X\times Y)$, where $h_1={\rm pr}_2^*c_1(\mathcal{O}_{\mathbb{P}^r}(1))$, $h_2={\rm pr}_4^*c_1(\mathcal{O}_{\mathbb{P}^r}(1))$
$${\rm CH}(X\times \mathbb{P}^r\times Y\times \mathbb{P}^r)=\oplus_{0\leq i\leq r,0\leq j\leq r}h_1^ih_2^j {\rm CH}(X\times Y),$$
which gives in particular
\begin{eqnarray}
\label{eqdecmpgamma}
\Gamma_\phi=\sum_{i,j} h_1^ih_2^j\Gamma_{\phi,i,j} ,\,\,\Gamma'_{\phi}=\sum_{i,j} h_1^ih_2^j\Gamma'_{\phi,i,j},
\end{eqnarray}
with $\Gamma_{\phi,i,j}\in {\rm CH}(X\times Y),\,\Gamma'_{\phi,i,j}\in  {\rm CH}(Y\times X)$.
We obviously have  $\Gamma=\Gamma_{\phi,0,r},\,\,\Gamma'=\Gamma'_{\phi,0,r}$.
With the notation (\ref{eqdecmpgamma}), we have
 \begin{eqnarray}
\label{eqcomposgamma}
\Gamma'_{\phi}\circ \Gamma_\phi=\sum_{i,j,j'} h_1^ih_2^{j'} \Gamma'_{\phi,r-j,j'}\circ \Gamma_{\phi,i,j}\,\,{\rm in}\,\,{\rm CH}(X\times \mathbb{P}^r\times X\times \mathbb{P}^r).
\end{eqnarray}
while
$\Delta_{X\times \mathbb{P}^r}=\sum_{i+j=r}h_1^ih_2^j\Delta_X$ in ${\rm CH}(X\times \mathbb{P}^r\times X\times \mathbb{P}^r)$.
The fact that $\Gamma'_{\phi}\circ \Gamma_\phi-\Delta_{X\times \mathbb{P}^r}$ is rationally equivalent to a cycle supported via the first projection over a proper closed algebraic subset of $X\times \mathbb{P}^r$ then implies (by taking $i=0$, $j'=r$ in (\ref{eqcomposgamma}))
that the cycle
  \begin{eqnarray}
\label{eqcyclesuma}\sum_{j}\Gamma'_{\phi,r-j,r}\circ \Gamma_{\phi,0,j}-\Delta_X
\end{eqnarray}
 is supported via the first projection over  a proper closed algebraic subset of $X$.
We observe now that ${\rm dim}\,\Gamma=n+r$, $n={\rm dim}\,X$, so that $\Gamma_{\phi,0,j}$ for $j<r$ has dimension $<n$, hence does not dominate
$X$ via the first projection. It follows that $\sum_{j<r}\Gamma'_{\phi,r-j,r}\circ \Gamma_{\phi,0,j}$ does not dominate
$X$ via the first projection, so that the remaining term in (\ref{eqcyclesuma}) with $j=r$, namely
$$ \Gamma'_{\phi,0,r}\circ \Gamma_{\phi,0,r}-\Delta_X$$
is rationally equivalent to a cycle  supported over a proper closed algebraic subset of $X$ via the first projection. Exchanging $X$ and $Y$ concludes the proof.
\end{proof}
\begin{rema}{\rm In the sequel, we will use   a weaker version of Proposition \ref{propcharstabdecomp}, stating only the first decomposition in (\ref{eqbiratdecomp}). There is in this case no need to assume that $X$ and $Y$ are of the same dimension. Furthermore, as noticed by Shinder,  the proof can be then  made simpler  by observing that the stated existence property for
$\Gamma,\,\Gamma'$ holds for pairs of  birational varieties, and also for the pair $(X,\,X\times \mathbb{P}^r)$.}
\end{rema}
We now prove the following version of the specialization theorem for decomposition of the diagonal fist proved in
\cite{voisin}, and later  improved in \cite{CTPI}. We will say that a variety $Z$ has mild singularities if
there exists a desingularization morphism  $\tau : \widetilde{Z}\rightarrow Z$ which is ${\rm CH}_0$-universally trivial in the sense of \cite{CTPI}. This means that $\tau_*$ is an isomorphism on ${\rm CH}_0$ over  any field $K$ containing $k$.
The easiest way to make this condition satisfied is to ask that $\tau$ has  the following  property : for each (irreducible) subvariety $M\subset Z$, the induced morphism $  \widetilde{Z}_M\rightarrow M$ has generic  fiber  smooth rational  over
$k(M)$.
For example, ordinary quadratic  singularities in dimension $\geq 2$ are mild. We refer to \cite{ctpiizvestya} for a more general geometric interpretation of the mildness condition.
\begin{theo} \label{theodecompstab} (i)  Let $\mathcal{Y}\rightarrow B$ be a projective flat morphism of relative dimension $n$, where $B$ is smooth.
Assume the general fiber $\mathcal{Y}_b$ is smooth and  stably birational to a fixed smooth projective variety $Y$ of dimension $n$. Then, for any desingularization  $ \widetilde{\mathcal{Y}_0}$
of $\mathcal{Y}_0$,  there exist
codimension $n$ cycles $\Gamma\in {\rm CH}^n(\widetilde{\mathcal{Y}}_0\times Y)$, $\Gamma'\in {\rm CH}^n(Y\times \widetilde{\mathcal{Y}_0})$ such that
\begin{eqnarray}
\label{eqformcomp}\Gamma'\circ \Gamma=\Delta_{\widetilde{\mathcal{Y}}_0}+Z+Z'\,\,{\rm in}\,\,{\rm CH}^n( \widetilde{\mathcal{Y}_0}\times  \widetilde{\mathcal{Y}_0}),
\end{eqnarray}
where $Z$ is supported on $D\times  \widetilde{\mathcal{Y}_0}$ for some proper closed algebraic subset $D$ of
$\widetilde{\mathcal{Y}_0}$ and $Z'$ is supported over $ \widetilde{\mathcal{Y}_0}\times E$, where $E$ is the exceptional locus of $\tau$.

(ii)    If  the special fiber
$\mathcal{Y}_0$ has mild singularities, one can achieve for an adequate choice of desingularization
$\widetilde{\mathcal{Y}}_0$ that $Z'=0$ in (\ref{eqformcomp}).
\end{theo}
\begin{proof}  (i)  We can assume $B$ is a smooth curve.
 By assumption and  using Proposition \ref{propcharstabdecomp}, there exist  for a general point $t\in B$ a
 divisor $D_t\subset \mathcal{Y}_t$ and codimension $n$ cycles
 $\Gamma_t\in {\rm CH}^n(\mathcal{Y}_t\times Y),\,\,\Gamma'_t\in {\rm CH}^n(Y\times \mathcal{Y}_t )$,  such that
 \begin{eqnarray}
 \label{eqcompogen} \Gamma'_t\circ \Gamma_t =\Delta_{\mathcal{Y}_t}+Z_t\,\,{\rm in}\,\,{\rm CH}^n(\mathcal{Y}_t\times \mathcal{Y}_t),
 \end{eqnarray}
 where $Z_t$ is supported on $D_t\times \mathcal{Y}_t$.
By a countability argument for the Chow varieties parameterizing cycles in fibers, $D_t$ and  the cycles
$Z_t$, $\Gamma_t,\,\Gamma'_t$ can be constructed in families after a base change $B'\rightarrow B$.  We will denote
$\mathcal{Y}':=\mathcal{Y}\times_BB'$.
 This provides us with varieties and cycles
$$\mathcal{D}\subset \mathcal{Y},\,\,\mathcal{Z}\in{\rm CH}( \mathcal{D}\times_{B'}\mathcal{Y}'),\,\,\Gamma\in {\rm CH}( \mathcal{Y}'\times Y),\,\
\Gamma'\in {\rm CH}(Y\times \mathcal{Y}' )
$$
whose fiber at the general point $t\in B'$ satisfies (\ref{eqcompogen})
(see \cite{voisin} for the more details).
Restricting to the regular locus of the morphism $\mathcal{Y}'\rightarrow B'$, the composition in  (\ref{eqcompogen}) still  makes sense as a relative
composition because
 for $\Gamma\in {\rm CH}^n(U\times Y)$, $\Gamma'\in {\rm CH}^n(Y\times U')$, the composition
$\Gamma'\circ \Gamma$ is well-defined whenever $U$ and $Y$ are smooth, and $Y$ is projective. Furthermore, by specialization of rational equivalence,
(\ref{eqcompogen}) holds  in ${\rm CH}^n(\mathcal{Y}_{0,reg}\times \mathcal{Y}_{0})$ for any $0\in B$. Here, as we assumed $B$
(hence $B'$) is a curve, the divisor $\mathcal{D}$ can be assumed not to contain any component of  the fiber $\mathcal{Y}_0$, hence to
restrict to a proper divisor $D_0\subset \mathcal{Y}_0$.
Identifying $\mathcal{Y}_{0,reg}$ with
$\widetilde{\mathcal{Y}}_{0}\setminus E$, we get as well cycles
$$\widetilde{Z}_0,\,\,\widetilde{\Gamma}_0\in {\rm CH}^n(\widetilde{\mathcal{Y}}_{0}\times Y),\,\widetilde{\Gamma}'_0\in {\rm CH}^n(Y\times \widetilde{\mathcal{Y}}_{0})$$
with $\widetilde{Z}_0$ supported on $\widetilde{D}_0$
such that the equality
\begin{eqnarray}
\label{eqformcomppresque}\widetilde{\Gamma}'_0\circ\widetilde{ \Gamma}_0=\Delta_{\widetilde{\mathcal{Y}}_0}+\widetilde{Z}_0
\end{eqnarray}
holds in ${\rm CH}^n( (\widetilde{\mathcal{Y}_0}\setminus E)\times ( \widetilde{\mathcal{Y}_0}\setminus E))$.
It follows from the localization exact sequence that the cycle
$\widetilde{\Gamma}'_0\circ\widetilde{ \Gamma}_0-\Delta_{\widetilde{\mathcal{Y}}_0}-\widetilde{Z}_0\in {\rm CH}^n( \widetilde{\mathcal{Y}}_{0}\times \widetilde{\mathcal{Y}}_{0})$ is rationally equivalent to a cycle  supported on
$E\times  \widetilde{\mathcal{Y}}_{0} \cup  \widetilde{\mathcal{Y}}_{0}\times E$. This last cycle is the sum of a cycle
$Z_1$ supported on  $E\times  \widetilde{\mathcal{Y}}_{0}$ and a cycle $Z_2$ supported on $ \widetilde{\mathcal{Y}}_{0}\times E$.
We thus proved (\ref{eqformcomp}) with $\Gamma=\widetilde{ \Gamma}_0,\,\Gamma'=\widetilde{\Gamma}'_0$,
$Z=\widetilde{Z}_0+Z_1$, $Z'=Z_2$.

 (ii)  As in  \cite{CTPI}, and using the fact that (\ref{eqcompogen}) holds  in ${\rm CH}^n(\mathcal{Y}_{0,reg}\times \mathcal{Y}_{0})$, we observe that the cycle $Z'\in{\rm CH}_n(\widetilde{\mathcal{Y}}_0\times E)$ vanishes by construction in ${\rm CH}_n(U\times \mathcal{Y}_0)$ for some dense Zariski open subset $U$ of $\widetilde{\mathcal{Y}}_0$.
 On the other hand,  we can  work by assumption  with  the resolution
 $\tau: \widetilde{\mathcal{Y}}_0\rightarrow \mathcal{Y}_0$ for which  the morphism $\tau$ is universally
 ${\rm CH}_0$-trivial. It follows that
 the cycle $Z'$, seen over the generic point
 of $\widetilde{\mathcal{Y}}_0 $ as a $
 0$-cycle of $\widetilde{\mathcal{Y}}_0 $ defined  on the field $k(\mathcal{Y}_0)$,
 vanishes in ${\rm CH}_0((\widetilde{\mathcal{Y}}_0)_{k(\mathcal{Y}_0}))$. Hence $Z'$ vanishes in
 ${\rm CH}^n(U\times\widetilde{\mathcal{Y}}_0 )$
 for some  dense Zariski open set  $U$ of  $\widetilde{\mathcal{Y}}_0$.
 By the localization exact sequence, it is thus supported on $D\times  \widetilde{\mathcal{Y}}_0 $, where $D= \widetilde{\mathcal{Y}}_0\setminus U$, and thus can be absorbed in the term  $Z$.
\end{proof}
\begin{coro} \label{corodeg} Under the  same assumptions as in Theorem \ref{theodecompstab} (ii), assume that $H^i_{nr}(Y,A)=0$ for some integer $i $ and
abelian group $A$. Then $H^i_{nr}( \widetilde{\mathcal{Y}_0}, A)=0$.
\end{coro}
\begin{proof} We let both sides of  formula (\ref{eqformcomp})  with $Z'=0$ act on
 $H^i_{nr}( \widetilde{\mathcal{Y}_0}, A)$ (see  \cite{CTV} for a construction of the action). The action of
$ \Gamma'\circ \Gamma$ factors through $H^i_{nr}(Y,A)$ hence it is $0$. Moreover the diagonal acts by the identity map.
We thus conclude that
for any $\alpha\in H^i_{nr}( \widetilde{\mathcal{Y}_0}, A)$,
$$\alpha=Z^*\alpha
.$$
On the other hand, as $Z$ is supported on
$D\times \widetilde{\mathcal{Y}}_0$, the class $Z^*\alpha$  vanishes on $U\times X$, where $U:=X\setminus D$. Hence $\alpha_{\mid U}=0$,
which implies $\alpha=0$ by \cite{blochogus}.
\end{proof}

We are now in position to prove the following result.
\begin{theo}\label{theomainapendix} Let $\mathcal{Y}\rightarrow B$ be a projective flat morphism of relative dimension $n$, where $B$ is smooth, the generic fiber is smooth,
 and  the special fiber
$\mathcal{Y}_0$ has mild singularities.
Assume

(i)   the very general fiber $\mathcal{Y}_b$ satisfies $H^i_{nr}( {\mathcal{Y}_b}, A)=0$,

(ii)   $H^i_{nr}( \widetilde{\mathcal{Y}_0}, A)\not=0$
for some (equivalently, any) desingularization $ \widetilde{\mathcal{Y}_0}$ of $\mathcal{Y}_0$.

 Then
two very general fibers  $\mathcal{Y}_b$,  $\mathcal{Y}_{b'}$ are not stably birational.
\end{theo}
\begin{proof}  Fix one very general fiber $\mathcal{Y}_{b'}$ and denote  it by $Y$. We want to show that
the general fiber $\mathcal{Y}_b$ is not stably birational to $Y$. If it is, Corollary \ref{corodeg}  and the vanishing
$H^i_{nr}(\mathcal{Y}_{b'},A)=0$  given by (i) imply that  $H^i_{nr}( \widetilde{\mathcal{Y}_0}, A)=0$, contradicting  assumption  (ii).
\end{proof}

The following  variant of  Theorem \ref{theomainapendix}   is proved as above, using Corollary \ref{corodeg2} below  instead of Corollary \ref{corodeg}.
\begin{theo}  \label{theomainapendixvariat} Let $\mathcal{Y}\rightarrow B$ be a projective flat morphism of relative dimension $n$, where $B$ is smooth. Assume

(i)   the very general fiber $\mathcal{Y}_b$ satisfies $H^i_{nr}( {\mathcal{Y}_b}, A)=0$,

(ii) The central fiber admits a desingularization   $ \widetilde{\mathcal{Y}_0} $ with exceptional divisor $E=\cup_jE_j$ with $E_j$ smooth,
and  $ \widetilde{\mathcal{Y}_0} $  has a nonzero class $\alpha\in H^i_{nr}( \widetilde{\mathcal{Y}_0}, A)$ which vanishes on all the divisors $E_j$.

 Then
two very general fibers  $\mathcal{Y}_b$,  $\mathcal{Y}_{b'}$ are not stably birational.

\end{theo}
The proof    uses the following variant of Corollary \ref{corodeg} based on Schreieder's criterion \cite{schreieder1}.
\begin{coro}  \label{corodeg2}  Under the  same assumptions as in Theorem \ref{theodecompstab} (i), assume that $H^i_{nr}(Y,A)=0$ for some integer $i $ and
abelian group $A$, and that $\mathcal{Y}_0$ has a desingularization $ \widetilde{\mathcal{Y}_0}$ with exceptional divisor
$E=\cup_jE_j$ with $E_j$ smooth. Then  any unramified cohomology class
$\alpha\in H^i_{nr}( \widetilde{\mathcal{Y}_0}, A)$ which vanishes on each component $E_j$ of  $E$ is identically $0$.
\end{coro}
\begin{proof}  We let both sides of  formula (\ref{eqformcomp})  act on
 $H^i_{nr}( \widetilde{\mathcal{Y}_0}, A)$ (see  \cite{CTV} for a construction of the action). The action of
$ \Gamma'\circ \Gamma$ factors through $H^i_{nr}(Y,A)$ hence it is $0$ since this group is assumed to be $0$.
We  conclude as before that
for any $\alpha\in H^i_{nr}( \widetilde{\mathcal{Y}_0}, A)$,
$$\alpha=Z^*\alpha +{Z'}^*\alpha
.$$
If $\alpha$ vanishes on all  the components $E_j$ of the exceptional divisor $E$, we have ${Z'}^*\alpha=0$.  We thus have $\alpha=Z^*\alpha$ and we conclude as before that
$\alpha=0$.
\end{proof}
The  families to which  Theorem \ref{theomainapendix} applies are essentiall, in characteristic $0$,  all the families of weighted Fano hypersurfaces for which the stable irrationality has been proved by
a degeneration argument to a mildly singular member having a nonzero unramified cohomology class of degree $\leq 3$.
For example, we have
\begin{theo}\label{theoappli}  (i)  Two very general quartic or sextic double solids or quartic hypersurfaces of dimension $3$ or $4$  over $\mathbb{C}$ are not stably birational.

(ii) Two very general hypersurfaces over $\mathbb{C}$ of degree $\geq 5$ and dimension $n$ with
$5\leq n\leq 9$
 are not stably birational.

\end{theo}
\begin{proof}  The case (i) uses  Theorem \ref{theomainapendix}. We know by \cite{artinmumford} in case of quartic double solids, by Beauville
\cite{beauville} in case of sextic double solids, by Colliot-Th\'{e}l\`{e}ne-Pirutka
\cite{CTPI} in case of quartic threefolds and Schreieder \cite{schreieder} in the case of
quartic fourfolds,
that they admit degenerations with mild singularities having a nonzero unramified cohomology class
of degree $2$, which is given by a nonzero torsion class in $H^3_B(\widetilde{\mathcal{X}}_0,\mathbb{Z})$.  On the other hand, for all these classes of varieties,
the smooth member   $X$ does not have torsion in $H^3_B(X,\mathbb{Z})$.  Theorem \ref{theomainapendix} thus applies.

For  case (ii), we use Schreieder's degeneration, which in the numerical range above produces a desingularized central fiber with a nonzero
unramified cohomology class of degree  $3$ with torsion coefficients on the desingularized central fiber, vanishing on the exceptional divisor. The very general hypersurface  $X$ on the other hand has trivial unramified cohomology of degree $3$. Indeed, it is proved in
\cite{CTV} that such a class measures the defect of the Hodge conjecture for degree $4$ integral Hodge classes on $X$. But the smooth  hypersurface   of degree $\geq 5$ in $\mathbb{P}^{n+1}$ for $5\leq n\leq9 $ has no integral  Hodge class of degree $4$  not coming from
$\mathbb{P}^n$  by the  Lefschetz theorem on hyperplane sections.
\end{proof}
The reason we can not a priori extend Theorem  \ref{theomainapendix} to all hypersurfaces shown by Schreieder
\cite{schreieder}  not to be stably rational
is the fact that we do not know how to compute unramified cohomology of degree $\geq 4$ for general hypersurfaces, so we are not able
to check Condition (i)
in Theorem \ref{theomainapendix}. Note that the case of smooth hypersurfaces is covered by Shinder's main theorem.

%\begin{comment}

%Coll\`{e}ge de France
%3 rue d'Ulm, 75005 Paris, France
%
%claire.voisin@imj-prg.fr

%\end{comment}

\end{document}